\documentclass[reqno,12pt]{amsart}
\usepackage{amssymb,latexsym}
\usepackage{array,multirow,makecell}
\usepackage{mathtools}
\usepackage{IEEEtrantools}
\usepackage{amsmath}
\usepackage[T1]{fontenc}

\setcellgapes{1pt}
\makegapedcells
\newtheorem{theorem}{Theorem}[section]
\newtheorem{corollary}[theorem]{Corollary}
\newtheorem{lemma}[theorem]{Lemma}
\newtheorem{proposition}[theorem]{Proposition}
\newtheorem{korselt's criterion}[theorem]{Korselt's criterion}
\newtheorem{definition}[theorem]{Definition}

\numberwithin{equation}{section}
\begin{document}
\title[The $\mathbb{Q}$-Korselt Set of $\mathrm{pq}$ ]
{The $\mathbb{Q}$-Korselt Set of $\mathrm{pq}$ }

\author{Nejib Ghanmi}
\address[Ghanmi]{(1)Preparatory Institute of Engineering Studies, Tunis university, Tunisia.}
\address[]{\hspace{1.5cm}(2)   University College of Jammum, Department of Mathematics, Mekkah, Saudi Arabia.}

\email{naghanmi@uqu.edu.sa\; and\; neghanmi@yahoo.fr}

\thanks{}



\begin{abstract}
Let $N$ be a positive integer,  $\mathbb{A}$ be a nonempty subset of $\mathbb{Q}$ and $\alpha=\dfrac{\alpha_{1}}{\alpha_{2}}\in \mathbb{A}\setminus \{0,N\}$. $\alpha$ is called an \emph{$N$-Korselt base} (equivalently $N$ is said an \emph{$\alpha$-Korselt number})
 if  $\alpha_{2}p-\alpha_{1}$ is a divisor of $\alpha_{2}N-\alpha_{1}$ for every prime  $p$ dividing $N$. The set   of all Korselt bases of $N$ in $\mathbb{A}$  is called the $\mathbb{A}$-Korselt set of $N$ and is simply denoted by $\mathbb{A}$-$\mathcal{KS}(N)$.

Let $p$ and $q$ be two distinct prime numbers. In this paper, we study the $\mathbb{Q}$-Korselt bases of $pq$, where we give in detail how to provide $\mathbb{Q}$-$\mathcal{KS}(pq)$. Consequently, we finish the incomplete characterization of the Korselt set of $pq$ over $\mathbb{Z}$  given in  ~\cite{Ghanmi}, by supplying  the set $\mathbb{Z}$-$\mathcal{KS}(pq)$  when $q <2p$.
\end{abstract}

%%% ----------------------------------------------------------------------
\maketitle
%%% ----------------------------------------------------------------------

\section{Introduction}
The notion of Korselt numbers ( or $\alpha$-Korselt number   with $\alpha\in\mathbb{Z}$) was introduced by Bouall\`egue-Echi-Pinch~\cite{BouEchPin,echi} as a generalization of Carmichael numbers~\cite{Bee,Car2}. Korselt numbers are defined simply as numbers which meet a generalized  Korselt criterion as follows.
\begin{definition}~\cite{Kor} Let $\alpha\in \mathbb{Z}\setminus \{0\}$. A positive integer $N$ is said to be an \emph{$\alpha$-Korselt number} (\emph{$K_{\alpha}$-number},
for short) if $N\neq \alpha$ and $p-\alpha$  divides $N-\alpha$ for
each prime divisor $p$ of $N$.
\end{definition}

Considerable progress has been made investigating Korselt numbers last years spatially in ~\cite{BouEchPin,Ghanmi,Ghanmi2,echi}. Many properties of Carmichael numbers are extended  for Korselt numbers. However,  many related questions remain open until now, such as the infinitude of Korselt numbers,  providing a complete characterization of the Korselt set of such a number, etc.
 Recently, Ghanmi  proposed in ~\cite{Ghanmi2,Ghanmi3} another generalization of Carmichael numbers;  he  extended the notion of Korselt numbers  to $\mathbb{Q}$  by stating the  following definitions.

\begin{definition}{\cite{Ghanmi2}}\label{def1}\rm Let $N\in \mathbb{N}\setminus\{0,1\}$,   $\alpha=\dfrac{\alpha_{1}}{\alpha_{2}}\in \mathbb{Q}\setminus \{0\}$  and $\mathbb{A}$ be a subset of $\mathbb{Q}$.
Then
\begin{enumerate}
  \item $N$ is said to be an \emph{$\alpha$-Korselt number\index{Korselt number}} (\emph{$K_{\alpha}$-number}, for
short), if $N\neq \alpha$ and $\alpha_{2}p-\alpha_{1}$ divides $\alpha_{2}N-\alpha_{1}$ for
every prime divisor $p$ of $N$.

\item By the \emph{$\mathbb{A}$-Korselt set}\index{Korselt set} of the number $N$ (or the Korselt set of $N$ over \emph{$\mathbb{A}$}) , we mean the set $\mathbb{A}$-$\mathcal{KS}(N)$ of
all $\beta\in \mathbb{A}\setminus\{0,N\}$ such that $N$ is a $K_{\beta}$-number.
  \item The cardinality of $\mathbb{A}$-$\mathcal{KS}(N)$ will be called the \emph{$\mathbb{A}$-Korselt
weight}\index{Korselt weight} of $N$; we denote it by $\mathbb{A}$-$\mathcal{KW}(N)$.

\end{enumerate}
\end{definition}

It's obvious  by this definition, that for $\alpha\in \mathbb{Z}\setminus \{0\}$ (i.e $\alpha_{2}=1$) we obtain the original $\alpha$-Korselt numbers introduced by Bouall\`egue-Echi-Pinch ~\cite{BouEchPin}.

\begin{definition}{\cite{Ghanmi3}}\label{def2}\rm  Let $N\in \mathbb{N}\setminus\{0,1\}$, $\alpha\in \mathbb{Q}$ and $\mathbb{B}$ be a subset of $\mathbb{N}$. Then
\begin{enumerate}
\item  $\alpha$ is called \emph{$N$-Korselt base\index{Korselt base}}(\emph{$K_{N}$-base}, for
short), if $N$ is a \emph{$K_{\alpha}$-number}.
\item By the \emph{$\mathbb{B}$-Korselt set}\index{Korselt base set} of the base $\alpha$ (or the Korselt set of the base $\alpha$ over \emph{$\mathbb{B}$}), we mean the set $\mathbb{B}$-$\mathcal{KS}(B(\alpha))$ of
all $M\in \mathbb{B}$ such that $\alpha$ is a
$K_{M}$-base.
\item The cardinality of  $\mathbb{B}$-$\mathcal{KS}(B(\alpha))$ will be called the \emph{$\mathbb{B}$-Korselt
weight}\index{Korselt  weight} of the base $\alpha$; we denote it by $\mathbb{B}$-$\mathcal{KW}(B(\alpha))$.
\end{enumerate}
\end{definition}

The set of all $\alpha$-Korselt numbers when $\alpha$ varies in  $\mathbb{Q}$, is called  the $\mathbb{Q}$-Korselt numbers (or  rational Korselt numbers or the set of  Korselt numbers over $\mathbb{Q}$ ). The set of all $N$-Korselt bases in $\mathbb{Q}$ when $N$ varies in  $\mathbb{N}$, is called the $\mathbb{N}$-Korselt bases in $\mathbb{Q}$ (or   $\mathbb{N}$-Korselt rational bases or  the set of  Korselt rational bases  over $\mathbb{N}$).

As we know, it's  not easy in general to determine the Korselt set of a given number $N$ over $\mathbb{Z}$;  even for the simple case when $N=pq$ with $p$ and $q$ are two distinct prime numbers. This difficulty is mentioned  in ~\cite{Ghanmi}  by Ghanmi and Echi, where they  characterized the $\mathbb{Z}$-Korselt set of $pq$ and left the case when $q<2p$ without solution. Fortunately, this problem will be solved in our present work; in fact, the set $\mathbb{Z}$-$\mathcal{KS}(pq)$ when $q<2p$, will be completely determined   after  studying  the Korselt set of $pq$ over $\mathbb{Q}$.

In this paper, and for  given distinct prime numbers $p$ and $q$, we will discuss in Section $2$, the belonging sets  of  $\alpha_{1}$  and $\alpha_{2}$ in $\mathbb{Z}$ for which $\alpha=\dfrac{\alpha_{1}}{\alpha_{2}}$ is a Korselt rational base of $pq$. In Section $3$ and by some results given in Section $2$, we will characterize completely the Korselt rational set of $pq$. Furthermore, this allowed us to deduce immediately the $\mathbb{Z}$-Korselt set of $pq$ when $q<2p$.

For all the following let $\alpha=\dfrac{\alpha_{1}}{\alpha_{2}}\in\mathbb{Q}$, $p<q$ be  two primes, $N=pq$ and $i,j,s,t$ be the integers given by the Euclidean divisions of
$q$ and $\alpha_{1}$ by $p$ : $q=ip+s$ and $\alpha_{1}=jp+t$ with $s \in \{1, \ldots, p-1\}$, $t \in \{0,1, \ldots, p-1\}$.

For $\alpha=\dfrac{\alpha_{1}}{\alpha_{2}}\in\mathbb{Q}$, we will suppose without loss of generality that $\alpha_{2}\in \mathbb{N}\setminus\{0\}$, $\alpha_{1}\in \mathbb{Z}$ and $\gcd(\alpha_{1},\alpha_{2})=1$.

As the case of $\alpha\in\mathbb{Z}$ (i.e. $\alpha_{2}=1$) is discussed in~\cite{Ghanmi}, our attention will center  in all this paper, only on the case when $\alpha_{2}\geq2$.

\section{Properties of  $pq$-Korselt rational bases }

\begin{proposition}{\cite{Ghanmi2}}\label{encadr1}
  Let $\alpha=\dfrac{\alpha_{1}}{\alpha_{2}}\in\mathbb{Q}$   such that $\gcd(\alpha_{1},N)=1$.
   If $\alpha$ is a $K_{N}$-base, then the following inequalities hold.
    \begin{equation*} q-p+1\leq \alpha \leq q+p-1.\end{equation*}
\end{proposition}

Now, we give informations about the Korselt set of $pq$ over $\mathbb{Q}$ when $\gcd(\alpha_{1},N)\neq 1$.
\begin{proposition}{\cite[Proposition $2.4$]{Ghanmi2}}\label{kornum}
  Suppose that $N$ is a $K_{\alpha}$-number with $\gcd(\alpha_{1},N)\neq 1$. Then the following assertions hold.
 \begin{enumerate}
   \item [1)]If $\alpha \in \mathbb{Z}$ $($ i.e. $\alpha_{2}=1; \alpha=\alpha_{1})$, then  $q \nmid \alpha, \ \ p\mid\alpha$ and
  $$\alpha\in\left\{\left\lfloor\dfrac{q}{p}\right\rfloor p , \left\lceil\dfrac{q}{p}\right\rceil p \right\}.$$

  \item [2)] If $\alpha \in \mathbb{Q}\setminus \mathbb{Z}$, then $\dfrac{q}{p}\leq \alpha \leq q+p-1$.

  \end{enumerate}
    \end{proposition}

\begin{proposition}\label{encadr3}

 Let $\alpha=\dfrac{\alpha_{1}}{\alpha_{2}}\in\mathbb{Q}$ be such that $\gcd(\alpha_{1},N)=1$.

  If $\alpha$ is a $K_{N}$-base, then the following properties hold.

   \begin{enumerate}
       \item If $q<2p$, then $1\leq\alpha_{2}\leq p$.
       \item If $2p<q<3p$, then $1\leq\alpha_{2}<3$.
       \item If $3p<q<4p$, then $1\leq\alpha_{2}<2$.
       \item If $q>4p$, then  $\alpha_{2}=1$  $($i.e. $\alpha\in\mathbb{Z})$.
   \end{enumerate}

\end{proposition}
\begin{proof}
By definition, $\alpha$ is a $K_{N}$-base if and only if

$$(S_{1})\hspace{1cm} \left\{\begin{array}{rrl}
   \alpha_{2}p-\alpha_{1}  & \mid & p(q-1) \\
   \alpha_{2}q-\alpha_{1}  & \mid & q(p-1)\\
   \end{array}
  \right. $$

 Since $\gcd(\alpha_{1},p)=\gcd(\alpha_{1},q)=1$ and $q=ip+s$, $(S_{1})$ is equivalent to
  $$(S_{2})\label{sy2}\hspace{1cm} \left\{\begin{array}{lrl}
   \alpha_{2}p-\alpha_{1}  & \mid & q-1=ip+s-1 \\
   \alpha_{2}(ip+s)-\alpha_{1}  & \mid & p-1\\
   \end{array}
  \right. $$

  This implies that
    \begin{equation}\label{eq1}
   -\alpha_{2}p+\alpha_{1}\leq ip+s-1
    \end{equation}

  and
  \begin{equation}\label{eq2}
  \alpha_{2}(ip+s)-\alpha_{1}\leq p-1.
   \end{equation}

  The sum of $\eqref{eq1}$ and $\eqref{eq2}$ gives
  \begin{equation}\label{eq3}
  \alpha_{2}((i-1)p+s)\leq (i+1)p+s-2,
  \end{equation}

  as $\alpha_{2}\geq 1$, it follows that
  \begin{equation}\label{eq4}
  \alpha_{2}(i-1)p\leq (i+1)p+s(1-\alpha_{2})-2<(i+1)p.
  \end{equation}
Two cases are to be discussed.
  \begin{enumerate}

       \item[a)]  If $q<2p$ hence  $q=p+s$ with $s\geq2$, (since $p = 2$; $q = 3$; $s = 1$ and $\alpha_{2} = 3$ would imply that $2\mid\alpha_{1}$), then  by $\eqref{eq3}$ we get \begin{equation*}\label{5}
       1\leq\alpha_{2}\leq \dfrac{2p-2}{s}+1\leq p.\end{equation*}
       \item[b)] Suppose that $i\geq2$. Then by  $\eqref{eq4}$, we obtain $1\leq\alpha_{2}< \dfrac{i+1}{i-1}$. Hence, the following subcases hold.
       \begin{itemize}
        \item If $2p<q<3p$ (i.e. $i=2$), then $\alpha_{2}< 3$.
        \item If $3p<q<4p$ (i.e. $i=3$),  then $\alpha_{2}< 2$.
        \item If $q>4p$ (i.e. $i\geq4$), then $\alpha_{2}=1$.
       \end{itemize}

   \end{enumerate}

  \end{proof}

  Each case in Proposition~\ref{encadr3} will be discussed separately in the  following two results.
  \begin{proposition}\label{encadr4}

Let $\alpha=\dfrac{\alpha_{1}}{\alpha_{2}}\in\mathbb{Q}$ be such that   $\gcd(\alpha_{1},N)=1$.
  If $\alpha$ is a $K_{N}$-base with  $q>2p$, then $\alpha\in\mathbb{Z}$.

\end{proposition}
\begin{proof}

 Let $\alpha=\dfrac{\alpha_{1}}{\alpha_{2}}\in \mathbb{Q}$-$\mathcal{KS}(N)$. First, let us note that  if $i\geq3$, then by Proposition~\ref{encadr3}  we have $\alpha\in\mathbb{Z}$. Hence, we may assume that  $i=2$.

  Substituting $\alpha_{1}=jp+t$ in  $(S_{2})$, we get

$$(S_{3})\hspace{1cm} \left\{\begin{array}{lrl}
   (\alpha_{2}-j)p-t& \mid & ip+s-1 \\
    (\alpha_{2}i-j)p+\alpha_{2}s-t & \mid & p-1 \\
   \end{array}
  \right. $$

which implies that

$$ \left\{\begin{array}{lrl}
   (j-\alpha_{2}-i)p  & \leq& s-t-1 \\
    \alpha_{2}s-t+1&\leq &(j-\alpha_{2}i+1)p \\
   \end{array}
  \right. $$
As $s-t-1\leq p-3$ and $-p+\alpha_{2}+2\leq \alpha_{2}s-t+1$, it follows that

  $$ \left\{\begin{array}{lrl}
   (j-\alpha_{2}-i)p  & \leq& p-3 \\
    -p+\alpha_{2}+2&\leq &(j-\alpha_{2}i+1)p \\
   \end{array}
  \right. $$

 hence
 $$ (S_{4})\hspace{1cm}\left\{\begin{array}{lrlll}
   j  & \leq&\alpha_{2}+i \\
    \alpha_{2}i-1&\leq &j \\
   \end{array}
  \right. $$
Now, let us show that for $i=2$, we have  $\alpha_{2}=1$ (i.e. $\alpha\in\mathbb{Z}$). If $i=2$, then  $\alpha_{2}< 3$ by Proposition~\ref{encadr3}.   Suppose that $\alpha_{2}=2$. Then by  $(S_{4})$, $j=3$ or $j=4$.

  \begin{itemize}
    \item If $j=3$, then $(S_{3})$ gives
    \begin{IEEEeqnarray}{rlCl}
&  p+t  & \mid & 2p+s-1 \label{eq8}
\\*[-0.625\normalbaselineskip]
\smash{\left\{
\IEEEstrut[5\jot]
\right.} \nonumber
\\*[-0.625\normalbaselineskip]
& p+2s-t  & \mid & p-1 \label{eq9}
\end{IEEEeqnarray}

  By $\eqref{eq8}$ we must have $2p+s-1=2(p+t)$, hence $s=2t+1$. Therefore $p+2s-t=p+3t+2>p-1$, which contradicts $\eqref{eq9}$.
    \item If $j=4$, then $(S_{3})$  becomes
    \begin{IEEEeqnarray}{rlCl}
& 2p+t  & \mid & 2p+s-1 \label{eq10}
\\*[-0.625\normalbaselineskip]
\smash{\left\{
\IEEEstrut[5\jot]
\right.} \nonumber
\\*[-0.625\normalbaselineskip]
& 2s-t  & \mid & p-1
\end{IEEEeqnarray}

  By $\eqref{eq10}$ we must have $2p+t = 2p+s-1 $, so that $t=s-1$. As $q=2p+s$, then $s$ is odd and so  $\alpha_{1}=jp+t=4p+s-1$ is even. But, since $\alpha_{2}=2$, we get a contradiction with the fact that $\gcd(\alpha_{1},\alpha_{2})=1$.
  \end{itemize}
Finally,  we conclude that $\alpha_{2}\neq2$. So, $\alpha_{2}=1$.

\end{proof}

\begin{proposition}\label{encadr5}

 Let  $\alpha=\dfrac{\alpha_{1}}{\alpha_{2}}\in\mathbb{Q}$  be such that $\gcd(\alpha_{1},N)=1$.
If $\alpha$ is a $K_{N}$-base with $q<2p$, then $\alpha_{2}\in\{j-1,j,j+1\}$.

\end{proposition}

\begin{proof}

If $\alpha$ is a $K_{N}$-base  such that $\gcd(\alpha_{1},p)=\gcd(\alpha_{1},q)=1$, $q=p+s$ and $\alpha_{1}=jp+t$, then by $(S_{2})$ we get

   $$(S_{5})\hspace{1cm} \left\{\begin{array}{lrl}
   (\alpha_{2}-j)p-t  & \mid & p+s-1 \\
  (\alpha_{2}-j)p+\alpha_{2}s-t  & \mid & p-1\\
   \end{array}
  \right. $$

  Hence
$$ \left\{\begin{array}{lrlll}
  -p-s+1 &\leq &(\alpha_{2}-j)p-t  & \leq& p+s-1 \\
   -p+1&\leq & (\alpha_{2}-j)p+\alpha_{2}s-t&\leq &p-1 \\
   \end{array}
  \right. $$

   so that
$$ \left\{\begin{array}{lrlll}
  -p-s+t+1 &\leq&(\alpha_{2}-j)p  & \leq& p+s+t-1 \\
   -p-\alpha_{2}s+t+1&\leq& (\alpha_{2}-j)p&\leq &p-\alpha_{2}s+t-1 \\
   \end{array}
  \right. $$

 This implies that

 \begin{equation*}\label{eq16}
 -p-s+t+1 \leq(\alpha_{2}-j)p  \leq p-\alpha_{2}s+t-1.\end{equation*}

 Since, in addition  \begin{equation*}\label{eq17}
 -2p+3=-p-(p-1)+1+1 \leq-p-s+t+1\end{equation*} and

\begin{equation*}\label{eq18}
p-\alpha_{2}s+t-1\leq p-1+p-1-1=2p-3,\end{equation*} it follows that

$$ \left\{\begin{array}{rll}
-2p <-2p+3 &\leq&(\alpha_{2}-j)p  \\
   (\alpha_{2}-j)p&\leq &2p-3<2p \\
   \end{array}
  \right. $$

  hence \begin{equation*}\label{eq19}
  -2 <\alpha_{2}-j<2.\end{equation*} So, we deduce that  $\alpha_{2}\in\{j-1,j,j+1\}$.

\end{proof}
The previous proposition leads us to discuss separately in the  next three lemmas, each case of  $\alpha_{2}\in\{j-1,j,j+1\}$  in order to fully determine the $K_{N}$-base $\alpha$.
For the rest of this paper, let us define for an integer $m$ the set $ Div(m-1)=\{d_{m}\in\mathbb{N}^{*}; \, d_{m} \mid (m-1) \}$.

\begin{lemma}\label{encadr6}

 Let $\alpha=\dfrac{\alpha_{1}}{\alpha_{2}}\in\mathbb{Q}$ be such that $\gcd(\alpha_{1},N)=1$.
 Suppose that $q<2p$, $j\geq2$  and $\alpha_{1}=(\alpha_{2}+1)p+t$ (i.e. $\alpha_{2}=j-1$). Then the following assertions hold.

 \begin{enumerate}
   \item  If $\alpha$ is a $K_{N}$-base then $\alpha_{2}$ is odd and  $1\leq\alpha_{2}\leq p$.
   \item $\alpha$ is a $K_{N}$-base if and only if there exist  $\varepsilon \in \{-1,1\}$ and an even integer $ d_{p}\in Div(p-1)$ such that $t=s-1$ and $\alpha_{2}=\dfrac{q-1-\varepsilon d_{p}}{s}$.
 \end{enumerate}

\end{lemma}

\begin{proof}
\begin{enumerate}

\item First, let us show  that $p\neq2$. Since $\alpha$ is a $K_{N}$-base with  $\gcd(\alpha_{1},p)=\gcd(\alpha_{1},q)=1$, $q=p+s$ and $\alpha_{1}=(\alpha_{2}+1)p+t$, we get by $(S_{5})$
\begin{IEEEeqnarray}{rlCl}
& p+t  & \mid & p+s-1 \label{eq22}
\\*[-0.625\normalbaselineskip]
\smash{(S_{6})\hspace{1cm}\left\{
\IEEEstrut[5\jot]
\right.} \nonumber
\\*[-0.625\normalbaselineskip]
& p+t-\alpha_{2}s  & \mid & p-1 \label{eq23}
\end{IEEEeqnarray}

By $\eqref{eq22}$ we must have $p+t =p+s-1 $, hence $t=s-1$.

 Suppose by contradiction that $p=2$. As $1\leq s\leq p-1=1$, it follows that $s=1$ hence $t=0$ and so $\alpha_{1}=(\alpha_{2}+1)p$. Therefore,   $p=2\mid\alpha_{1}$, which  contradicts the fact that $\gcd(\alpha_{1},p)=1$.

Now, since $p\geq3$ and $q=p+s$, it follows that $s$ is even and so $t$ is odd. Further, we claim that $\alpha_{2}$ is odd, else if $\alpha_{2}$ is even and as $t$ is odd then $\alpha_{1}=(\alpha_{2}+1)p+t$   will be even, hence $2\mid\gcd(\alpha_{1},\alpha_{2})=1$, which is impossible.

Proposition ~\ref{encadr3} (1)  implies that $1\leq \alpha_{2} \leq p$.

\item Since $p\geq3$ and $s$ is even, we obtain  by  $(S_{6})$ the equivalence: $\alpha\in \mathbb{Q}$-$\mathcal{KS}(N)$ if and only if
 there exist $\varepsilon \in \{-1,1\}$ and an even integer $d_{p}\in Div(p-1)$ such that

$$ \left\{\begin{array}{lrl}
   t=s-1\\
   p-1-(\alpha_{2}-1)s   = \varepsilon d_{p}\\
   \end{array}
  \right. $$

  This is equivalent to  $t=s-1$ and $\alpha_{2}= \dfrac{q-1-\varepsilon d_{p}}{s}$.

\end{enumerate}

\end{proof}

\begin{lemma}\label{encadr7}

 Let $\alpha=\dfrac{\alpha_{1}}{\alpha_{2}}\in\mathbb{Q}$ be such that $\gcd(\alpha_{1},N)=1$.
 Suppose that $q<2p$, $j\geq1$  and $\alpha_{1}=\alpha_{2}p+t$ \ \ $($i.e. $\alpha_{2}=j)$. Then,
 $\alpha$ is a $K_{N}$-base if and only if  there exist $\varepsilon \in \{-1,1\}$, $d_{p}\in Div(p-1)$ and $d_{q}\in Div(q-1)$ such that \, $t=d_{q}\leq p-1$ and $\alpha_{2}=\dfrac{d_{q}+\varepsilon d_{p}}{s}>0$.

\end{lemma}

\begin{proof}
Suppose that $\alpha\in \mathbb{Q}$-$\mathcal{KS}(N)$ with  $\gcd(\alpha_{1},p)=\gcd(\alpha_{1},q)=1$, $\alpha_{1}=\alpha_{2}p+t$ and $q=p+s$. Then $(S_{5})$ becomes

$$ \left\{\begin{array}{lrl}
   t  & \mid & q-1 \\
   \alpha_{2}s-t  & \mid & p-1\\
   \end{array}
  \right. $$

It follows that, $\alpha$ is a $K_{N}$-base if and only if
 there exist $\varepsilon \in \{-1,1\}$, $d_{p}\in Div(p-1)$ and $d_{q}\in Div(q-1)$ such that

$$\left\{\begin{array}{lrl}
   t   = d_{q}\\
   \alpha_{2}s-t   = \varepsilon d_{p} \\
   \end{array}
  \right. $$

which is equivalent to  $t= d_{q}$ and $\alpha_{2}=\dfrac{d_{q}+\varepsilon d_{p}}{s}>0$.

\end{proof}

\begin{lemma}\label{encadr8}

 Let $\alpha=\dfrac{\alpha_{1}}{\alpha_{2}}\in\mathbb{Q}$ be such that $\gcd(\alpha_{1},N)=1$.
 Suppose that $q<2p$ and $\alpha_{1}=(\alpha_{2}-1)p+t$ \,  $($i.e. $\alpha_{2}=j+1)$. Then, $\alpha$ is a $K_{N}$-base if and only if there exist $d_{p}\in Div(p-1)$, $d_{p}\geq2$ and $d_{q}\in Div(q-1)$ such that \, $t=p-d_{q}>0$ and $\alpha_{2}=\dfrac{d_{p}- d_{q}}{s}>0$.

\end{lemma}

\begin{proof}
If $\alpha\in \mathbb{Q}$-$\mathcal{KS}(N)$ such that  $\gcd(\alpha_{1},p)=\gcd(\alpha_{1},q)=1$, $q=p+s$ and  $\alpha_{1}=(\alpha_{2}-1)p+t$,  then $(S_{5})$ becomes

$$ \left\{\begin{array}{lrl}
   p-t  & \mid & q-1 \\
   p-t+\alpha_{2}s  & \mid & p-1\\
   \end{array}
  \right. $$

Therefore, $\alpha$ is a $K_{N}$-base if and only if
 there exist   $d_{p}\in Div(p-1)$ and $d_{q}\in Div(q-1)$ such that

$$\left\{\begin{array}{lrl}
   p-t   = d_{q}\\
   \alpha_{2}s+p-t   = d_{p}\geq2 \\
   \end{array}
  \right. $$

  this is equivalent to  $t= p-d_{q}>0$ and $\alpha_{2} =\dfrac{d_{p}- d_{q}}{s}>0$.

\end{proof}

By the next result, we give bounds for $\alpha_{2}$ when  $\alpha=\dfrac{\alpha_{1}}{\alpha_{2}}$ is a $K_{N}$-base where  $\gcd(\alpha_{1},N)=q$.
\begin{proposition}\label{encadr9}
   Let  $\alpha=\dfrac{\alpha_{1}}{\alpha_{2}}\in\mathbb{Q}$-$\mathcal{KS}(N)$ be such that $\gcd(\alpha_{1},N)=q$. Then the following assertions  hold.

   \begin{enumerate}
       \item If $q<2p$, then $1\leq\alpha_{2}\leq \dfrac{p(p+2)}{2}$.
       \item If $q>2p$, then $1\leq\alpha_{2}< \dfrac{p(i+1)}{i-1}$.
   \end{enumerate}

\end{proposition}
\begin{proof}

  Suppose that $\alpha\in\mathbb{Q}$-$\mathcal{KS}(N)$  such that $\gcd(\alpha_{1},N)=q$, so that $\gcd(\alpha_{1},p)=1$. Then, by
$(S_{1})$ we get

$$  \left\{\begin{array}{rrl}
   \alpha_{2}p-\alpha_{1}  & \mid & q-1\\
   \alpha_{2}q-\alpha_{1}  & \mid & q(p-1)\\
   \end{array}
  \right. $$
Therefore
\begin{IEEEeqnarray}{rlCl}
&  -\alpha_{2}p+\alpha_{1} & \leq & q-1 \label{eq28}
\\*[-0.625\normalbaselineskip]
\smash{\left\{
\IEEEstrut[5\jot]
\right.} \nonumber
\\*[-0.625\normalbaselineskip]
& \alpha_{2}q-\alpha_{1}  & \leq& q(p-1) \label{eq29}
\end{IEEEeqnarray}

 hence, by summing  $\eqref{eq28}$ and $\eqref{eq29}$ we get
   \begin{equation}\label{eq30}
   \alpha_{2}(q-p)   \leq qp-1<qp. \end{equation}
 Two cases are to be discussed.
  \begin{enumerate}
  \item If $i=1$, then $q=p+s$. We distinguish two subcases.
  \begin{itemize}
    \item If $s$ is even (i.e $s\geq2$), then by  $\eqref{eq30}$ we obtain
  \begin{equation*}\label{eq33}
  \alpha_{2}s=\alpha_{2}(q-p)<qp=(p+s)p.\end{equation*}
  As $s\geq2$, it follows that
  \begin{equation}\label{eq34}
  \alpha_{2}<\dfrac{(p+s)p}{s}\leq\dfrac{(p+2)p}{2}.\end{equation}
    \item If $s$ is odd (i.e $s=1$, $q=3$ and $p=2$), then easily we can prove that $\alpha\in \left\{\dfrac{3}{2}, \dfrac{9}{4}\right\}$, and so  $\alpha_{2}\leq\dfrac{(p+2)p}{2}$.
  \end{itemize}
\item  If $q=ip+s>2p$, then by $\eqref{eq30}$ we get
  \begin{equation*}\label{eq31}\alpha_{2}(i-1)p <\alpha_{2}((i-1)p+s)=\alpha_{2}(q-p)  < qp<(i+1)p^{2}.\end{equation*}
hence  $  \alpha_{2}<\dfrac{(i+1)p}{i-1}$.

\end{enumerate}
\end{proof}
Now, by the following two results, we determine all $K_{N}$-bases $\alpha\in\mathbb{Q}$  such that   $\gcd(\alpha_{1},N)=q$  and $q<2p$.
\begin{proposition}\label{encadr10}
Let  $\alpha=\dfrac{\alpha_{1}}{\alpha_{2}}\in\mathbb{Q}$-$\mathcal{KS}(N)$  such that   $\gcd(\alpha_{1},N)=q$. Suppose that  $\alpha_{1}=\alpha_{1}^{'}q$ and $q<2p$. Then the following assertions hold.

   \begin{enumerate}
       \item $\alpha_{1}^{'}<\alpha_{2}$.
       \item  If $\alpha_{1}^{'}=1$ then $\alpha_{2}\in\{2,3\}$, both for $p = 3$; $q = 5$, otherwise we have
       \begin{enumerate}
         \item $\alpha_{2}=2$ if and only if $(p-s)\mid (2s-1)$.
         \item $\alpha_{2}=3$ if and only if $p>2$  and $s=\dfrac{p+1}{2}$.
       \end{enumerate}

   \end{enumerate}
\end{proposition}
\begin{proof}
\begin{enumerate}
  \item Let  $\alpha=\dfrac{\alpha_{1}}{\alpha_{2}}\in\mathbb{Q}$-$\mathcal{KS}(N)$ be such that $\gcd(\alpha_{1},p)=1$, $\alpha_{1}=\alpha_{1}^{'}q$ and $q=p+s$ .
 Then $(S_{1})$ is equivalent to

\begin{IEEEeqnarray}{rlCl}
& (\alpha_{2}-\alpha_{1}^{'})p-\alpha_{1}^{'}s & \mid & p+s-1 \label{eq35}
\\*[-0.625\normalbaselineskip]
\smash{(S_{7})\hspace{1cm}\left\{
\IEEEstrut[5\jot]
\right.} \nonumber
\\*[-0.625\normalbaselineskip]
& \alpha_{2}-\alpha_{1}^{'}  & \mid & p-1 \label{eq36}
\end{IEEEeqnarray}

  First, by $\eqref{eq36}$, we have  $\alpha_{1}^{'}\neq\alpha_{2}$.  Now,  we claim that  $\alpha_{1}^{'}<\alpha_{2}$. If is not the case and as $\alpha_{1}^{'}>0$, then   we get $(\alpha_ {1}^{'}-\alpha_{2})p+\alpha_{1}^{'}s\geq p+s>0$, which contradicts $\eqref{eq35}$.

  \item Suppose that $\alpha_{1}^{'}=1$. Then by $\eqref{eq35}$  we obtain $$(\alpha_{2}-1)p-s \leq p+s-1.$$

  As $s\leq p-1$, it follows that $(\alpha_{2}-2)p \leq 2s-1<2p,$
   hence $\alpha_{2}<4$. But, since $\alpha_{2}>\alpha_{1}^{'}=1$, it yields that $\alpha_{2}\in\{2,3\}$.

   \begin{enumerate}
    \item If $\alpha_{2}=2$, then since $\alpha_{1}^{'}=1$,  $(S_{7})$ will be  reduced to $(p-s) \mid (p+s-1)=p-s+2s-1,$ which is equivalent to  $(p-s)\mid (2s-1)$.
    \item Now, suppose that $\alpha_{2}=3$. Then $(S_{7})$ is simply reduced to $2\mid p-1$ and $(2p-s)\mid (p+s-1).$ Since $p+s-1<2(2p-s)$, it follows that $(S_{7})$ is equivalent to $ p>2$ and $2p-s=p+s-1$, so $s=\dfrac{p+1}{2}$.
  \end{enumerate}

  \end{enumerate}
\end{proof}
Now, in the  next result  where $q<2p$ and $\alpha_{1}=\alpha_{1}^{'}q\geq2q$, we consider   the Euclidean division of $\alpha_{2}$ by $\alpha_{1}^{'}$: $\alpha_{2}=m\alpha_{1}^{'}+r$ with $m\geq1$ and $1\leq r\leq\alpha_{1}^{'}-1$.

\begin{lemma}\label{encadr11}
Let  $\alpha=\dfrac{\alpha_{1}}{\alpha_{2}}\in\mathbb{Q}$-$\mathcal{KS}(N)$  be such that   $\gcd(\alpha_{1},N)=q$. Suppose that  $\alpha_{1}=\alpha_{1}^{'}q\geq2q$ and $q<2p$. Then the following  hold.

   \begin{enumerate}
       \item $m\in\{1,2\}$.
       \item If  $m=2$, then $r=1$, $\alpha_{1}^{'}=\dfrac{s-1}{p-s}$ and $\alpha_{2}=\dfrac{q-2}{p-s}$ .
       \item If  $m=1$, then there exist  $d_{p}\in Div(p-1)$, $d_{q}\in Div(q-1)$ and  $\varepsilon \in \{-1,1\}$ such that $\alpha_{1}^{'}=\dfrac{pd_{p}-\varepsilon d_{q}}{s}$ and $\alpha_{2}=\dfrac{qd_{p}-\varepsilon d_{q}}{s}$.

   \end{enumerate}
\end{lemma}
\begin{proof}

\begin{enumerate}
  \item Let $\alpha=\dfrac{\alpha_{1}}{\alpha_{2}}\in\mathbb{Q}$-$\mathcal{KS}(N)$ be such that   $\alpha_{1}=\alpha_{1}^{'}q$, $q=p+s$ and $\gcd(\alpha_{1},p)=1$. Then by $\eqref{eq35}$ we obtain

\begin{equation}\label{eq37}
(\alpha_{2}-\alpha_{1}^{'})p-\alpha_{1}^{'}s  \mid  p+s-1.\end{equation}

Since $\alpha_{2}=m\alpha_{1}^{'}+r$ and $s\leq p-1$, it follows that
\begin{equation*}\label{eq38}
((m-2)\alpha_{1}^{'}+r)p+\alpha_{1}^{'}(p-s)=(\alpha_{2}-\alpha_{1}^{'})p-\alpha_{1}^{'}s\leq p+s-1<2p.\end{equation*}
As $\alpha_{1}^{'}(p-s)>0$, we get $(m-2)\alpha_{1}^{'}+r\leq1$,  hence \begin{equation}\label{eq39}
(m-2)\alpha_{1}^{'}\leq 1-r\leq0.\end{equation}
This implies that  $m\leq2$.

  \item Assume that $m=2$. Then by $\eqref{eq39}$ we obtain $r=1$ and so $\alpha_{2}=2\alpha_{1}^{'}+1$. Therefore $\eqref{eq37}$ becomes
 $ p+\alpha_{1}^{'}(p-s)\mid p+s-1$.  But since  $2(p+\alpha_{1}^{'}(p-s))>2p> p+s-1$, it follows that $p+\alpha_{1}^{'}(p-s)= p+s-1$. Hence $\alpha_{1}^{'}= \dfrac{s-1}{p-s}$ and  so $\alpha_{2}=\dfrac{q-2}{p-s}.$
  \item Suppose that $m=1$. Then  $\alpha_{2}=\alpha_{1}^{'}+r$ with $1\leq r\leq\alpha_{1}^{'}-1$. Therefore $(S_{7})$ becomes
$$ \left\{\begin{array}{rrl}
 rp-\alpha_{1}^{'}s & \mid & q-1\\
   r & \mid & p-1\\
   \end{array}
  \right. $$

   So, $\alpha\in \mathbb{Q}$-$\mathcal{KS}(N)$ if and only if
 there exist $d_{q}\in Div(q-1)$, $d_{p}\in Div(p-1)$ and $\varepsilon \in \{-1,1\}$ such that
  $$ \left\{\begin{array}{rrl}
 rp-\alpha_{1}^{'}s & =& \varepsilon d_{q}\\
   r & =& d_{p}\\
   \end{array}
  \right. $$
  This is equivalent to  $\alpha_{1}^{'}=\dfrac{pd_{p}-\varepsilon d_{q}}{s}$ and $\alpha_{2}=\dfrac{qd_{p}-\varepsilon d_{q}}{s}$.

\end{enumerate}

\end{proof}
The remaining case where $\gcd(\alpha_{1},N)=q$ and $q>2p$ will be treated in the following lemma.
\begin{lemma}\label{encadr12}
Let $\alpha=\dfrac{\alpha_{1}}{\alpha_{2}}\in\mathbb{Q}$ be such that $\gcd(\alpha_{1},N)=q$.
Suppose that $\alpha_{1}=\alpha_{1}^{'}q$ and $q>2p$. Then, $\alpha$ is a $K_{N}$-base if and only if
 there exist $\varepsilon \in \{-1,1\}$, $d_{p}\in Div(p-1)$ and $d_{q}\in Div(q-1)$ such that $\alpha_{1}^{'}=\dfrac{pd_{p}-\varepsilon d_{q}}{q-p}$ and $\alpha_{2}=\dfrac{qd_{p}-\varepsilon d_{q}}{q-p}$.

\end{lemma}
\begin{proof}

Let  $\alpha\in\mathbb{Q}$-$\mathcal{KS}(N)$ be such that $\gcd(\alpha_{1},p)=1$, $\alpha_{1}=\alpha_{1}^{'}q$ and $q=ip+s$.  Then $(S_{1})$ is equivalent to
 \begin{IEEEeqnarray}{rlCl}
& \alpha_{2}p-\alpha_{1}^{'}q & \mid & q-1 \label{eq41}
\\*[-0.625\normalbaselineskip]
\smash{(S_{8})\hspace{1cm}\left\{
\IEEEstrut[5\jot]
\right.} \nonumber
\\*[-0.625\normalbaselineskip]
& \alpha_{2}-\alpha_{1}^{'}  & \mid & p-1\label{eq42}
\end{IEEEeqnarray}

 Let us show that $\alpha_{2}>\alpha_{1}^{'}$.  First, it's clear by $\eqref{eq42}$ that $\alpha_{2}\neq\alpha_{1}^{'}$.
 Now, suppose by contradiction that $\alpha_{2}<\alpha_{1}^{'}$. Then by $\eqref{eq41}$ we have
  $\alpha_{1}^{'}q-\alpha_{2}p\leq q-1$, hence $ q<(\alpha_{2}+1)q-\alpha_{2}p\leq\alpha_{1}^{'}q-\alpha_{2}p\leq q-1$,
which impossible.

Since $\alpha_{2}-\alpha_{1}^{'} >0$, it follows by $(S_{8})$ that  $\alpha$ is a $K_{N}$-base if and only if there exist $d_{p}\in Div(p-1)$,  $d_{q}\in Div(q-1)$ and $\varepsilon \in \{-1,1\}$  such that

  $$ \left\{\begin{array}{rrl}
   \alpha_{2}p-\alpha_{1}^{'}q  & = &\varepsilon d_{q}\\
   \alpha_{2}-\alpha_{1}^{'}  & =& d_{p}\\
   \end{array}
  \right. $$

This is equivalent to $\alpha_{1}^{'}=\dfrac{pd_{p}-\varepsilon d_{q}}{q-p}$ and $\alpha_{2}=\dfrac{qd_{p}-\varepsilon d_{q}}{q-p}$.

\end{proof}
Now, the case where  $\gcd(\alpha_{1},N)=p$ will be studied in the next two results.
\begin{proposition}\label{encadr13}

   Let $\alpha=\dfrac{\alpha_{1}}{\alpha_{2}}\in\mathbb{Q}$-$\mathcal{KS}(N)$ be such that $\gcd(\alpha_{1},N)=p$. Suppose that $\alpha_{1}=\alpha_{1}^{'}p$. Then the following assertions hold.
\begin{enumerate}
  \item If $\alpha_{2}=1$ then $\alpha_{1}^{'}\in\{i,i+1\}$.
  \item If $\alpha_{2}\geq2$ then $i\alpha_{2}+1\leq\alpha_{1}^{'}\leq (i+1)\alpha_{2}-1.$
\end{enumerate}

\end{proposition}
\begin{proof}

  Let  $\alpha\in\mathbb{Q}$-$\mathcal{KS}(N)$ be such that $\gcd(\alpha_{1},q)=1$, $\alpha_{1}=\alpha_{1}^{'}p$ and $q=ip+s$. Then, $(S_{1})$ is equivalent to
  \begin{IEEEeqnarray}{rlCl}
& \alpha_{2}-\alpha_{1}^{'} & \mid & ip+s-1
\\*[-0.625\normalbaselineskip]
\smash{(S_{9})\hspace{1cm}\left\{
\IEEEstrut[5\jot]
\right.} \nonumber
\\*[-0.625\normalbaselineskip]
& (i\alpha_{2}-\alpha_{1}^{'})p+ \alpha_{2}s & \mid & p-1 \label{eq54}
\end{IEEEeqnarray}

We claim that \begin{equation}\label{eq55}
i\alpha_{2}-\alpha_{1}^{'}\leq0. \end{equation} Indeed,  if is not the case, we get $(i\alpha_{2}-\alpha_{1}^{'})p+ \alpha_{2}s\geq p+\alpha_{2}s>p$, which contradicts $\eqref{eq54}$. Also, we claim that \begin{equation}\label{eq56}
\alpha_{1}^{'}-(i+1)\alpha_{2}\leq 0.\end{equation} Indeed, if is not the case, we obtain  \begin{equation*}\label{eq57}
(\alpha_{1}^{'}-i\alpha_{2})p- \alpha_{2}s=(\alpha_{1}^{'}-(i+1)\alpha_{2})p+\alpha_{2}(p-s)\geq p+\alpha_{2}(p-s)>p,\end{equation*}

which contradicts $\eqref{eq54}$. Now, we consider the two following cases.
\begin{enumerate}
\item If $\alpha_{2}=1$, then by $\eqref{eq55}$ and $\eqref{eq56}$ we get
\begin{equation*}\label{eq58}
i\leq\alpha_{1}^{'}\leq i+1\end{equation*}

\item  Assume that $\alpha_{2}\geq2$. We claim that $i\alpha_{2}-\alpha_{1}^{'}\neq0$. Indeed if not, then  $i\alpha_{2}=\alpha_{1}^{'}$. Hence $1=\gcd(\alpha_{1},\alpha_{2})=\gcd(\alpha_{1}^{'},\alpha_{2})=\alpha_{2}\geq2$, which is impossible. Therefore by $\eqref{eq55}$ we get  \begin{equation}\label{eq59}
    i\alpha_{2}+1\leq\alpha_{1}^{'}.\end{equation}

    Similarly, we claim that $\alpha_{1}^{'}-(i+1)\alpha_{2}\neq0$.
 Hence by $\eqref{eq56}$  we obtain \begin{equation}\label{eq60}
 \alpha_{1}^{'}\leq(i+1)\alpha_{2}-1.\end{equation}

Combining $\eqref{eq59}$ and $\eqref{eq60}$,  we get \begin{equation}\label{eq61}
i\alpha_{2}+1\leq\alpha_{1}^{'}\leq(i+1)\alpha_{2}-1.\end{equation}

\end{enumerate}
\end{proof}

\begin{lemma}\label{encadr14}
Let $\alpha=\dfrac{\alpha_{1}}{\alpha_{2}}\in\mathbb{Q}$ be such that $\gcd(\alpha_{1},N)=p$. Assume that
$\alpha_{1}=\alpha_{1}^{'}p$. Then, $\alpha$ is a $K_{N}$-base if and only if there exist $\varepsilon\in\{-1,1\}$, $d_{p}\in Div(p-1)$ and $d_{q}\in Div(q-1)$  such that $\alpha_{1}^{'}=\dfrac{qd_{q}+\varepsilon d_{p}}{q-p}$ and $\alpha_{2}=\dfrac{pd_{q}+\varepsilon d_{p}}{q-p}$.
\end{lemma}
\begin{proof}
 Let  $\alpha\in\mathbb{Q}$-$\mathcal{KS}(N)$ be such that $\gcd(\alpha_{1},q)=1$, $\alpha_{1}=\alpha_{1}^{'}p$ and $q=ip+s$. As  by $\eqref{eq61}$ $\alpha_{1}^{'}=i\alpha_{2}+\alpha_{3}$ with $1\leq\alpha_{3}<\alpha_{2}$ for $\alpha_{2}\geq 2$, (resp. with $0\leq\alpha_{3}\leq\alpha_{2}$ for $\alpha_{2}=1$),  $(S_{9})$ becomes
 \begin{IEEEeqnarray}{rlCl}
& (i-1)\alpha_{2}+\alpha_{3} & \mid & ip+s-1=q-1
\\*[-0.625\normalbaselineskip]
\smash{\left\{
\IEEEstrut[5\jot]
\right.} \nonumber
\\*[-0.625\normalbaselineskip]
& \alpha_{3}p-\alpha_{2}s & \mid & p-1
\end{IEEEeqnarray}

 It follows that, $\alpha$ is a $K_{N}$-base if and only if there exist  $\varepsilon\in\{-1,1\}$,  $d_{p}\in Div(p-1)$ and $d_{q}\in Div(q-1)$  such that
  $$ \left\{\begin{array}{rrl}
 (i-1)\alpha_{2}+\alpha_{3}  & =& d_{q}\\
   \alpha_{2}s-\alpha_{3}p   & =& \varepsilon d_{p}\\
   \end{array}
  \right. $$
 This is equivalent to $\alpha_{2}=\dfrac{pd_{q}+\varepsilon d_{p}}{q-p}$,  $\alpha_{3}=\dfrac{sd_{q}+\varepsilon(1-i)d_{p}}{q-p}$ and so $\alpha_{1}^{'}=\dfrac{qd_{q}+\varepsilon d_{p}}{q-p}$.

\end{proof}
In the rest of this section, we will discuss the case where $pq\mid\alpha_{1}$.
\begin{proposition}\label{encadr15}
Let $\alpha=\dfrac{\alpha_{1}}{\alpha_{2}}\in\mathbb{Q}$-$\mathcal{KS}(N)$ be such that $\alpha_{1}=\alpha_{1}^{''}pq$. Then the following assertions hold.

   \begin{enumerate}
       \item $\alpha_{1}^{''}<\alpha_{2}$.
       \item $q\leq4p-3$.
   \end{enumerate}

\end{proposition}
\begin{proof}

 Let $\alpha=\dfrac{\alpha_{1}}{\alpha_{2}}\in\mathbb{Q}$-$\mathcal{KS}(N)$ be such that  $\alpha_{1}=\alpha_{1}^{''}pq$. Then $(S_{1})$ becomes

$$ \left\{\begin{array}{rrl}
    \alpha_{2}p-\alpha_{1}^{''}pq=\alpha_{2}p-\alpha_{1}   & \mid & p(q-1) \\
     \alpha_{2}q-\alpha_{1}^{''}pq=\alpha_{2}q-\alpha_{1} & \mid & q(p-1)\\
   \end{array}
  \right. $$

which is equivalent to
\begin{IEEEeqnarray}{rlCl}
& \alpha_{1}^{''}q -\alpha_{2}& \mid & q-1 \label{eq64}
\\*[-0.625\normalbaselineskip]
\smash{(S_{10})\hspace{1cm}\left\{
\IEEEstrut[5\jot]
\right.} \nonumber
\\*[-0.625\normalbaselineskip]
& \alpha_{1}^{''}p -\alpha_{2} & \mid & p-1 \label{eq65}
\end{IEEEeqnarray}

\begin{enumerate}
  \item By $\eqref{eq64}$ we have $ \alpha_{1}^{''}q -\alpha_{2}\leq q-1$. As in addition $q>1$, it follows that $\alpha_{1}^{''}-1<(\alpha_{1}^{''}-1)q \leq\alpha_{2}-1$ for $\alpha_{1}^{''}>1$ which implies that  $\alpha_{1}^{''}<\alpha_{2}$. If $\alpha_{1}^{''}=1$, then since $ pq\neq\alpha=\dfrac{pq}{\alpha_{2}}$ with $\alpha_{2}\geq 1$ we get $\alpha_{2}>1=\alpha_{1}^{''}$.
  \item By $(S_{10})$ we can write
  $$ \left\{\begin{array}{rrl}
   \alpha_{1}^{''}q -\alpha_{2}& \mid & \alpha_{1}^{''}(q-1)=\alpha_{1}^{''}q -\alpha_{2}+\alpha_{2}-\alpha_{1}^{''}\\
   \alpha_{1}^{''}p -\alpha_{2} & \mid & p-1\\
   \end{array}
  \right. $$
  hence

  $$ \left\{\begin{array}{rrl}
  \alpha_{1}^{''}q -\alpha_{2}& \mid &\alpha_{2}-\alpha_{1}^{''}\\
   \alpha_{1}^{''}p-\alpha_{2} & \mid &p-1\\
   \end{array}
  \right. $$

  This implies that

   $$ \left\{\begin{array}{rrl}
  \alpha_{1}^{''}q -\alpha_{2}& \leq &\alpha_{2}-\alpha_{1}^{''}\\
   -p+1 & \leq &\alpha_{1}^{''}p-\alpha_{2}\\
   \end{array}
  \right. $$
therefore
\begin{IEEEeqnarray}{rlCl}
& \dfrac{(q+1)\alpha_{1}^{''}}{2}& \leq &\alpha_{2} \label{eq66}
\\*[-0.625\normalbaselineskip]
\smash{\left\{
\IEEEstrut[5\jot]
\right.} \nonumber
\\*[-0.625\normalbaselineskip]
& \alpha_{2} & \leq &(\alpha_{1}^{''}+1)p-1 \label{eq67}
\end{IEEEeqnarray}

So, by combining $\eqref{eq66}$ and $\eqref{eq67}$, we get $$\dfrac{(q+1)\alpha_{1}^{''}}{2}\leq\alpha_{2}\leq(\alpha_{1}^{''}+1)p-1.$$
Hence $\alpha_{1}^{''}(q-2p+1)\leq2(p-1)$. As in addition  $\alpha_{1}^{''}\geq 1$, we get $q-2p+1\leq\alpha_{1}^{''}(q-2p+1)\leq2(p-1)$.
Thus $q\leq4p-3$.

\end{enumerate}

\end{proof}

\begin{lemma}\label{encadr16}
Let $\alpha=\dfrac{\alpha_{1}}{\alpha_{2}}\in\mathbb{Q}$ be such that $\alpha_{1}=\alpha_{1}^{''}pq$. Suppose  that

$2p<q<3p$. Then, $\alpha$ is a $K_{N}$-base if and only if one of the following assertions is verified.
\begin{enumerate}
\item $q=\dfrac{9p-5}{4}$ or $q=3p-2$, $\alpha_{1}^{''}=1$ and $\alpha_{2}=\dfrac{3p-1}{2}$.
\item $s+1\mid q-1$, $\alpha_{1}^{''}=1$ and $\alpha_{2}=2p-1$.
\item $\dfrac{s+1}{2}\mid p-1$, $\alpha_{1}^{''}=1$ and $\alpha_{2}=\dfrac{q+1}{2}$.
\end{enumerate}
 \end{lemma}
\begin{proof}

 Suppose that $\alpha\in\mathbb{Q}$-$\mathcal{KS}(N)$ .  Then by $\eqref{eq64}$ and $\eqref{eq65}$ we get respectively
     \begin{equation}\label{eq68}
     \alpha_{1}^{''}q -\alpha_{2}\leq q-1 \end{equation} and
     \begin{equation}\label{eq69}
     -\alpha_{1}^{''}p +\alpha_{2}\leq p-1. \end{equation}

    Adding $\eqref{eq68}$ to $\eqref{eq69}$, we obtain \begin{equation}\label{eq70}
    \alpha_{1}^{''}(q-p)\leq q+p-2.  \end{equation}
   Since $q=2p+s$, it follows that \begin{equation*}\label{eq71}
   \alpha_{1}^{''}\leq \dfrac{q+p-2}{q-p}=\dfrac{3(p+s)-2s-2}{p+s}<3.\end{equation*}

   Therefore, $\alpha_{1}^{''}\in\{1,2\}$.   Let us show that $\alpha_{1}^{''}=1$.

   Suppose  by contradiction that $\alpha_{1}^{''}=2$. Then by $\eqref{eq68}$ and $\eqref{eq69}$ we get
   $ 2p+s+1=q+1\leq\alpha_{2}\leq 3p-1$. Hence, the Euclidean division of $\alpha_{2}$ by $p$ is $ \alpha_{2}=2p+u $ with $s+1\leq u\leq p-1$.

    As by  $\eqref{eq64}$ we have $2(2q -\alpha_{2})-(q-1)=2(p-u)+3s+1>0$ and $(2q -\alpha_{2})\mid (q-1)$, then we must have $2q -\alpha_{2}=q-1$
that is $\alpha_{2}=q+1$.  But, $2\mid\gcd(q+1,2)=\gcd(\alpha_{2},\alpha_{1}^{''})=1$, which is not true.  So, $\alpha_{1}^{''}=1$.

  Now, let us show that $p<\alpha_{2}\leq 2p-1$. As  $\alpha_{1}^{''}=1$, $(S_{10})$ becomes
  \begin{IEEEeqnarray}{rlCl}
& q -\alpha_{2}& \mid & q-1 \label{eq73}
\\*[-0.625\normalbaselineskip]
\smash{(S_{11})\hspace{1cm}\left\{
\IEEEstrut[5\jot]
\right.} \nonumber
\\*[-0.625\normalbaselineskip]
& \alpha_{2}-p & \mid & p-1 \label{eq74}
\end{IEEEeqnarray}

By $\eqref{eq74}$, it follows immediately that $\alpha_{2}\leq 2p-1$. Hence, we can write $\alpha_{2}= kp+u$ with $k\in\{0,1\}$ and $0<u\leq p-1$.

We claim that $k\neq0$ (i.e. $k=1$ and so $p<\alpha_{2}$). Indeed, if not, then  $\alpha_{2}= u$ and as $2u<2p<q$, we get
\begin{equation*}\label{eq75}
 2(q -\alpha_{2})=q+(q-2u)\geq q+1>q-1.\end{equation*}
Therefore, by $\eqref{eq73}$ we must have $q -\alpha_{2}=q-1$, so that $\alpha_{2}=1$. This implies that $\alpha=pq=N$, which is not possible  by definition.
Thus, \begin{equation}\label{eq76}
p<\alpha_{2}\leq 2p-1<q  \end{equation}

Since  $q -\alpha_{2}>0$ and $\alpha_{2}-p>0$ by $\eqref{eq76}$  it follows by $(S_{11})$ that $\alpha$  is a $K_{N}$-base if and only if   there exist  $d_{p}\in Div(p-1)$ and $d_{q}\in Div(q-1)$ such that
\begin{IEEEeqnarray}{rlCl}
& q -\alpha_{2}& = & d_{q} \label{eq77}
\\*[-0.625\normalbaselineskip]
\smash{\left\{
\IEEEstrut[5\jot]
\right.} \nonumber
\\*[-0.625\normalbaselineskip]
& \alpha_{2}-p & = & d_{p} \label{eq78}
\end{IEEEeqnarray}

Let us show that   $d_{p}=\dfrac{p-1}{k_{1}}$ and $d_{q}=\dfrac{q-1}{k_{2}}$ with $(k_{1}, k_{2})\neq (1,1)$ and ($k_{1}\in \{1,2\}$ or $k_{2}\in \{1,2\}$).

Adding $\eqref{eq77}$ to $\eqref{eq78}$, we get \begin{equation}\label{eq79}
q-p=d_{q}+d_{p}.\end{equation}

\begin{itemize}
  \item If $d_{p}=p-1$ and $d_{q}=q-1$, then by $\eqref{eq79}$ we obtain \begin{equation*}\label{eq80}
   p+s=q-p=d_{q}+d_{p}=q+p-2=3p+s-2.\end{equation*} Hence $p=1$, which is not possible.
  \item Now, assume that $d_{p}\leq\dfrac{p-1}{3}$ and $d_{q}\leq\dfrac{q-1}{3}$. Then by $\eqref{eq79}$
  $p+s=q-p=d_{q}+d_{p}\leq\dfrac{p+q-2}{3}$. Hence  $2s\leq-2$, which is not true.

\end{itemize}

So, we conclude that $d_{p}\in\{p-1, \dfrac{p-1}{2}\}$ or $d_{q}\in\{q-1, \dfrac{q-1}{2}\}$. First, we claim that $d_{q}\neq\ q-1$, indeed if $d_{q}=\ q-1$ it follows by $\eqref{eq77}$ and $\eqref{eq78}$ that $\alpha_{2}=1$ and so $\alpha=pq=N$ which is not possible.
 This yields to consider the  following cases:

\begin{enumerate}
  \item If $d_{p}=\dfrac{p-1}{2}$, then by $\eqref{eq78}$ we have $\alpha_{2}=\dfrac{3p-1}{2}$. Therefore $q-\alpha_{2}=\dfrac{p+2s+1}{2}\mid q-1=2p+s-1$ by $\eqref{eq77}$. Knowing that $4(q-\alpha_{2})-(q-1)=3s+3>0$ and $q-\alpha_{2}-(q-1)=1-\alpha_{2}<0$, it follows that $q-1=k(q-\alpha_{2})$ with $k\in\{2,3\}$ .  This implies that $k=2$ (i.e. $d_{q}=\dfrac{q-1}{2}$)  hence $s=p-2$ and $q=3p-2$ or $k=3$ (i.e. $d_{q}=\dfrac{q-1}{3}$) which gives $s=\dfrac{p-5}{4}$ and  $q=\dfrac{9p-5}{4}$.
 \item Suppose that $d_{p}=p-1$. Then by $\eqref{eq78}$ we obtain $\alpha_{2}=2p-1$, this yields by $\eqref{eq77}$  that
    $q-\alpha_{2}=s+1\mid q-1$.
    \item Now, assume that $d_{q}=\dfrac{q-1}{2}$. Then by $\eqref{eq77}$ we obtain $\alpha_{2}=\dfrac{q+1}{2}$, this yields by $\eqref{eq78}$  that
    $\alpha_{2}-p=\dfrac{s+1}{2}\mid p-1$.
\end{enumerate}

  This ends the forward direction of the proposition. The backward direction is clearly satisfied.

\end{proof}

\begin{lemma}\label{encadr17}
Let $\alpha=\dfrac{\alpha_{1}}{\alpha_{2}}\in\mathbb{Q}$ be such that $\alpha_{1}=\alpha_{1}^{''}pq$. Suppose that

$3p<q<4p$. Then, $\alpha$ is a $K_{N}$-base if and only if  $q=4p-3$, $\alpha_{1}^{''}=1$ and $\alpha_{2}=2p-1$.
\end{lemma}
\begin{proof}

 Suppose that $\alpha\in\mathbb{Q}$-$\mathcal{KS}(N)$, $\alpha_{1}=\alpha_{1}^{''}pq$ and $3p<q<4p$.

 Let us show that $\alpha_{1}^{''}=1$.  As $q=3p+s$ we get by $\eqref{eq70}$
  \begin{equation*}\label{eq85}
  \alpha_{1}^{''}\leq \dfrac{q+p-2}{q-p}=\frac{2(2p+s)-s-2}{2p+s}<2. \end{equation*}
    Thus, $\alpha_{1}^{''}=1$.

  Now, since $\alpha_{1}^{''}=1$ and with the same proof of $\eqref{eq76}$, we may write  \begin{equation*}\label{eq86}
  p<\alpha_{2}=p+u\leq 2p-1<q.\end{equation*}
  Hence,  $(S_{10})$ becomes
\begin{IEEEeqnarray}{rlCl}
& 2p+s-u=q -\alpha_{2}& \mid & q-1=3p+s-1 \label{eq87}
\\*[-0.625\normalbaselineskip]
\smash{\left\{
\IEEEstrut[5\jot]
\right.} \nonumber
\\*[-0.625\normalbaselineskip]
& u=\alpha_{2}-p & \mid & p-1 \label{eq89}
\end{IEEEeqnarray}

As $3(2p+s-u)=3p+s-1+3(p-u)+2s+1>3p+s-1$ and $2p+s-u<3p+s-1$, it follows by $\eqref{eq87}$ that
$2(2p+s-u)=3p+s-1$. Hence, $u=\dfrac{p+s+1}{2}$.

By $\eqref{eq89}$, we have $ u=\dfrac{p+s+1}{2}\mid (p-1)$. Since, in addition, $2u=p+s+1>  p-1$, it follows that $\dfrac{p+s+1}{2}=p-1$. Hence $s=p-3$; so that $q=4p-3$ and $\alpha_{2}=2p-1$.

 This ends the forward direction of the proposition.  The backward direction is clearly satisfied.
 \end{proof}
 Now, for the  next two results, we consider $\alpha_{2}=kp+u$  with  $1\leq u\leq p-1$; the Euclidian division of $\alpha_{2}$ by $p$.

\begin{proposition}\label{encadr18}
Let $\alpha=\dfrac{\alpha_{1}}{\alpha_{2}}\in\mathbb{Q}$-$\mathcal{KS}(N)$ be such that $\alpha_{1}=\alpha_{1}^{''}pq$. If $q<2p$, then the following assertions hold.

   \begin{enumerate}
       \item $\alpha_{1}^{''}\leq p+1$.
       \item $\alpha_{1}^{''}\in\{k,k+1\}$.
   \end{enumerate}
\end{proposition}
\begin{proof}
Suppose that $\alpha\in\mathbb{Q}$-$\mathcal{KS}(N)$ and $q=p+s$.
\begin{enumerate}
  \item  By $\eqref{eq70}$ we have $ \alpha_{1}^{''}\leq \dfrac{q+p-2}{q-p}=\dfrac{2p+s-2}{s}$. Hence, two cases are to be considered:
\begin{itemize}
  \item If $s=1$ and so $q=3, p=2$, then  $\alpha_{1}^{''}\leq 3=p+1$.
  \item If $s\geq2$, then $\alpha_{1}^{''}\leq \dfrac{2p+s-2}{s}\leq p$.
\end{itemize}
 \item   Substituting $q=p+s$ and $\alpha_{2}=kp+u$  in $\eqref{eq68}$, we get
    $(\alpha_{1}^{''}-k)p+\alpha_{1}^{''}s -u=\alpha_{1}^{''}q -\alpha_{2}\leq q-1=p+s-1$,
  therefore  $(\alpha_{1}^{''}-k-1)p+(\alpha_{1}^{''}-1)s \leq u-1\leq p-2.$ Since, in addition, $\alpha_{1}^{''}\geq 1$ hence $(\alpha_{1}^{''}-1)s\geq0$, it follows that $\alpha_{1}^{''}-k-1\leq 0 $, that is  $\alpha_{1}^{''}\leq k+1 $.

  Also,  by $\eqref{eq69}$ we have  $(\alpha_{1}^{''}-k)p -u=\alpha_{1}^{''}p -\alpha_{2}\geq -p+1$, hence $(\alpha_{1}^{''}-k+1)p\geq u+1\geq 2$. Therefore, $\alpha_{1}^{''}-k+1\geq 1 $, that is  $\alpha_{1}^{''}\geq k.$  So, we deduce that $\alpha_{1}^{''}\in\{k,k+1\}$.
\end{enumerate}

\end{proof}
Now, the next two lemmas deal with the two cases of $\alpha_{1}^{''}\in\{k,k+1\}$.
\begin{lemma}\label{encadr19}
Let $\alpha=\dfrac{\alpha_{1}}{\alpha_{2}}\in\mathbb{Q}$ be such  that $\alpha_{1}=kpq$ and $q<2p$. Then, $\alpha$ is a $K_{N}$-base if and only if  there exist $\varepsilon\in\{-1,1\}$, $d_{p}\in Div(p-1)$ and $d_{q}\in Div(q-1)$ such that $k=\dfrac{d_{p}+\varepsilon d_{q}}{s}>0$ and $\alpha_{2}=\dfrac{qd_{p}+\varepsilon p d_{q}}{s}>0$.

\end{lemma}
\begin{proof}
Replacing $\alpha_{2}=kp+u $, $q=p+s$ and $\alpha_{1}=kpq$ in $(S_{10})$, we obtain

$$ \left\{\begin{array}{rrl}
   ks-u & \mid & q-1\\
   -u & \mid & p-1\\
   \end{array}
  \right. $$

 Hence,  $\alpha$ is a $K_{N}$-base if and only if  there exist  $d_{q}\in Div(q-1)$, $d_{p}\in Div(p-1)$ and $\varepsilon\in\{-1,1\}$ such that
  $$ \left\{\begin{array}{rrl}
 ks-u  & =& \varepsilon d_{q}\\
   u & =& d_{p}\\
   \end{array}
  \right. $$
 This is equivalent to  $k=\dfrac{d_{p}+\varepsilon d_{q}}{s}$ and $\alpha_{2}=\dfrac{qd_{p}+\varepsilon p d_{q}}{s}$.

\end{proof}
\begin{lemma}\label{encadr30}
Let $\alpha=\dfrac{\alpha_{1}}{\alpha_{2}}\in\mathbb{Q}$ be such that  $\alpha_{1}=(k+1)pq$ and $q<2p$. Then, $\alpha$ is a $K_{N}$-base if and only if  there exist $d_{p}\in Div(p-1)$ and $d_{q}\in Div(q-1)$ such that $k+1=\dfrac{d_{q}-d_{p}}{s}>0$ and $\alpha_{2}=\dfrac{pd_{q}-qd_{p}}{s}>0$.
\end{lemma}
\begin{proof}
Substituting $\alpha_{2}=kp+u $, $q=p+s$ and $\alpha_{1}=(k+1)pq$ in $(S_{10})$, we get

$$ \left\{\begin{array}{rrl}
   (k+1)s+p-u & \mid & q-1\\
  p-u & \mid & p-1\\
   \end{array}
  \right. $$

 Therefore,  $\alpha$ is a $K_{N}$-base if and only if  there exist  $d_{p}\in Div(p-1)$ and $d_{q}\in Div(q-1)$ such that
  $$ \left\{\begin{array}{rrl}
 (k+1)s+p-u  & =& d_{q}\\
   p-u & =& d_{p}\\
   \end{array}
  \right. $$
 This is equivalent to  $k+1=\dfrac{d_{q}-d_{p}}{s}>0$ and $\alpha_{2}=\dfrac{pd_{q}-qd_{p}}{s}>0$.

\end{proof}

\section{The Korselt Set of $pq$}

First, let us give the Korselt set of $pq$ over $\mathbb{Z}$ when $q<2p$.
\begin{corollary}\label{encadr31}
Let $\alpha\in\mathbb{Z}$  such that $\gcd(\alpha,N)=1$. Assume that  $q<2p$. Then, $\alpha$ is a $K_{N}$-base if and only if  there exist  $d_{q}\in Div(q-1)$ and $\varepsilon \in \{-1,1\}$ such that $\alpha= p+\varepsilon d_{q}$ and $ s-\varepsilon d_{q}$ divides $p-1$.
\end{corollary}
\begin{proof}

Let $\alpha\in\mathbb{Z}$ (i.e. $\alpha_{2}=1$) such that $\gcd(\alpha,N)=1$.
Suppose that  $q<2p$.  Then,  we have  by Lemmas~\ref{encadr6}, ~\ref{encadr7} and ~\ref{encadr8}, respectively, $\alpha$ is a $K_{N}$-base if and only if

\begin{itemize}
  \item $\alpha=q+p-1$,
  \item $\alpha=p+d_{q}$ such that $(s-d_{q})\mid (p-1)$, $d_{q}\leq p-1$.
  \item $\alpha=p-d_{q}>0$ such that $(s+d_{q})\mid (p-1)$.
\end{itemize}
This is equivalent to  the existence of  $d_{q}\in Div(q-1)$ and $\varepsilon\in \{-1,1\}$ such that $\alpha=p+\varepsilon d_{q}>0$ and $ (s-\varepsilon d_{q})\mid(p-1)$.

\end{proof}
By Corollary ~\ref{encadr31}, we derive immediately the following.

\begin{theorem}\label{structure1}Let $p< q$ be two prime numbers,
$N=pq$ and $q=p+s$ such that $1\leq s\leq p-1$. Then

$$\mathbb{Z}\text{-}\mathcal{KS}(N)= \left( \bigcup\limits_{\substack{d_{q}\mid q-1\\
                                                            \varepsilon\in\{-1,1\}}}
         \left\{p+\varepsilon d_{q} \, ; \, s-\varepsilon d_{q}\mid p-1 \right\}\right)\cup\{2p \, ; \; p-s\mid p-1\}.$$

\end{theorem}

Now, before giving the mean result in this paper, we set for  given   distinct prime numbers $p$ and $q$ the following:

 $$\mathcal{A}_{p,q}=  \bigcup\limits_{\substack{d_{q}\mid q-1,\, d_{p}\mid p-1\\
                                                          \varepsilon\in\{-1,1\}}}
         \left\{\dfrac{qd_{q}+\varepsilon pd_{p}}{d_{q}+\varepsilon d_{p}} \, ; \, (q-p) \mid (d_{q}+\varepsilon d_{p}),\, d_{q}+\varepsilon d_{p}\neq 0\right\}$$

 $$\mathcal{B}_{p,q}=  \bigcup\limits_{\substack{d_{q}\mid q-1,\, d_{p}\mid p-1\\
                                                          \varepsilon\in\{-1,1\}}}
         \left\{\dfrac{(pd_{p}-\varepsilon d_{q})q}{qd_{p}- \varepsilon d_{q}} \, ; \,  (q-p) \mid (qd_{p}- \varepsilon d_{q}),\, q-p <pd_{p}- \varepsilon d_{q} \right\}$$

$$\mathcal{C}_{p,q}=\bigcup\limits_{\substack{d_{q}\mid q-1,\, d_{p}\mid p-1\\
                                                          \varepsilon\in\{-1,1\}}}
           \left\{\dfrac{(qd_{q}+\varepsilon d_{p})p}{pd_{q}+\varepsilon d_{p}} \, ; \,  (q-p) \mid (pd_{q}+\varepsilon d_{p})\right\}$$

$$\mathcal{D}_{p,q}=\bigcup\limits_{\substack{d_{q}\mid q-1,\, d_{p}\mid p-1\\
                                                          \varepsilon\in\{-1,1\}}}
           \left\{\dfrac{(d_{p}+ \varepsilon d_{q})pq}{qd_{p}+ \varepsilon pd_{q}} \, ; \, (q-p) \mid (d_{p}+ \varepsilon d_{q}), \, d_{p}+ \varepsilon d_{q}\neq0 \right\}$$

\begin{theorem}[Structure of the Rational Korselt Set of $pq$]\label{structure1}

Let $p< q$ be two prime numbers,
$N=pq$ and $q=ip+s$ such that $1\leq s\leq p-1$. Then the following properties hold.
\begin{itemize}
    \item[$(1)$] If $q>4p$, then \begin{equation*}\label{eq93}
    \mathbb{Q}\text{-}\mathcal{KS}(N)=\mathcal{B}_{p,q}\cup\mathcal{C}_{p,q}\cup\{p+q-1\}.\end{equation*}
    \item[$(2)$] Assume that $3p<q<4p$. Then the following assertions hold.
    \begin{enumerate}
      \item [(a)] If $q=4p-3$, then \begin{equation*}\label{eq94}
      \mathbb{Q}\text{-}\mathcal{KS}(N)=\mathcal{B}_{p,q}\cup \mathcal{C}_{p,q}\cup\left\{q-p+1, p+q-1, \dfrac{pq}{2p-1}\right\}.\end{equation*}
      \item [(b)]If $q\neq4p-3$, then \begin{equation*}\label{eq95}
      \mathbb{Q}\text{-}\mathcal{KS}(N)= \mathcal{B}_{p,q}\cup \mathcal{C}_{p,q}\cup\{ p+q-1\}.\end{equation*}
    \end{enumerate}

    \item[$(3)$] Suppose that  $2p<q<3p$. Then the following conditions are satisfied.
    \begin{enumerate}
    \item [(i)] If $s+1$ divides $q-1$.
     \begin{enumerate}
       \item [(a)]If $s=\dfrac{p-5}{4}$ or $s=p-2$, then  \begin{equation*}\label{eq96}
       \mathbb{Q}\text{-}\mathcal{KS}(N)=\mathcal{B}_{p,q}\cup \mathcal{C}_{p,q}\cup\left\{3q-5p+3,q-p+1,p+q-1,\dfrac{2p+q-1}{2},\dfrac{2pq}{3p-1}, \dfrac{pq}{2p-1}, \dfrac{2pq}{q+1}\right\}.\end{equation*}
       \item [(b)] If   $s\neq \dfrac{p-5}{4}$ and $s\neq p-2$ then
       \begin{equation*}\label{eq97}
       \mathbb{Q}\text{-}\mathcal{KS}(N)=\mathcal{B}_{p,q}\cup \mathcal{C}_{p,q}\cup \left\{q-p+1,p+q-1,\dfrac{2p+q-1}{2},\dfrac{pq}{2p-1}, \dfrac{2pq}{q+1}\right\}.\end{equation*}
       \end{enumerate}
       \item [(ii)] If $s+1$ not dividing $q-1$, then \begin{equation*}\label{eq98}
       \mathbb{Q}\text{-}\mathcal{KS}(N)=\mathcal{B}_{p,q}\cup \mathcal{C}_{p,q}\cup \left\{p+q-1\right\}.\end{equation*}
     \end{enumerate}

        \item[$(4)$] Assume that $q<2p$. Then the following assertions hold.
        \begin{enumerate}
      \item [(i)] If $q=5$, then $$\mathbb{Q}\text{-}\mathcal{KS}(N)= \mathcal{A}_{p,q}\cup \mathcal{B}_{p,q}\cup \mathcal{C}_{p,q}\cup \mathcal{D}_{p,q}\cup \left\{\dfrac{q}{2},\dfrac{q}{3}\right\}.$$
       \item [(ii)]Suppose that $q\neq5$. Then the following subcases hold.
      \begin{enumerate}
      \item [(a)]If $p>2$ and $s=\dfrac{p+1}{2}$, then

      $\mathbb{Q}\text{-}\mathcal{KS}(N)= \mathcal{A}_{p,q}\cup \mathcal{B}_{p,q}\cup \mathcal{C}_{p,q}\cup \mathcal{D}_{p,q}\cup \left\{\dfrac{q}{3}\right\}.$
      \item [(b)]If $s=p-1$, then

      $\mathbb{Q}\text{-}\mathcal{KS}(N)= \mathcal{A}_{p,q}\cup \mathcal{B}_{p,q}\cup \mathcal{C}_{p,q}\cup \mathcal{D}_{p,q}\cup \left\{\dfrac{(s-1)q}{q-2},\dfrac{q}{2}\right\}.$
      \item[(c)] Assume $p-s$  divides $2s-1$ and $s\neq p-1$. Then

      $\mathbb{Q}\text{-}\mathcal{KS}(N)= \mathcal{A}_{p,q}\cup \mathcal{B}_{p,q}\cup \mathcal{C}_{p,q}\cup \mathcal{D}_{p,q}\cup \left\{\dfrac{q}{2}\right\}.$
      \item[(d)] If $p-s$  divides $s-1$, $s\neq\dfrac{p+1}{2}$ and $s\neq p-1$, then

      $\mathbb{Q}\text{-}\mathcal{KS}(N)= \mathcal{A}_{p,q}\cup \mathcal{B}_{p,q}\cup \mathcal{C}_{p,q}\cup \mathcal{D}_{p,q}\cup \left\{\dfrac{(s-1)q}{q-2}\right\}.$
      \item[(e)] If $p-s$  divides neither $s-1$ nor $2s-1$, then

      $\mathbb{Q}\text{-}\mathcal{KS}(N)= \mathcal{A}_{p,q}\cup \mathcal{B}_{p,q}\cup \mathcal{C}_{p,q}\cup \mathcal{D}_{p,q}.$

       \end{enumerate}
      \end{enumerate}
\end{itemize}
\end{theorem}
\begin{proof}For $N=pq$, let the sets
 $$\mathcal{A}=  \left\{\alpha=\dfrac{\alpha_{1}}{\alpha_{2}}\in\mathbb{Q}\text{-}\mathcal{KS}(N) \, ; \,\gcd(\alpha_{1},N)=1 \right\},$$
 $$\mathcal{B}=  \left\{\alpha=\dfrac{\alpha_{1}}{\alpha_{2}}\in\mathbb{Q}\text{-}\mathcal{KS}(N) \, ; \,\gcd(\alpha_{1},N)=q \right\},$$
 $$\mathcal{C}=  \left\{\alpha=\dfrac{\alpha_{1}}{\alpha_{2}}\in\mathbb{Q}\text{-}\mathcal{KS}(N) \, ; \,\gcd(\alpha_{1},N)=p \right\},$$
 $$\mathcal{D}=  \left\{\alpha=\dfrac{\alpha_{1}}{\alpha_{2}}\in\mathbb{Q}\text{-}\mathcal{KS}(N) \, ; \,\gcd(\alpha_{1},N)=pq \right\}.$$
 Clearly, we have $$\mathbb{Q}\text{-}\mathcal{KS}(N)=\mathcal{A}\cup\mathcal{B}\cup\mathcal{C}\cup\mathcal{D}.$$

First, note that if $q>2p$, then  by Lemmas~\ref{encadr12} and ~\ref{encadr14} we have respectively $\mathcal{C}=\mathcal{C}_{p,q} $ and $\mathcal{B}=\mathcal{B}_{p,q}$. Also, by Proposition~\ref{encadr4} $$\mathcal{A}=\left\{\alpha\in\mathbb{Z}\text{-}\mathcal{KS}(N)\, ; \,\gcd(\alpha,N)=1 \right\}.$$
\begin{enumerate}
  \item [$(1)$] If $q>4p$, then  we have by ~\cite[Theorem $14$]{Ghanmi} and Proposition~\ref{encadr15}, respectively, $\mathcal{A}=\left\{q+p-1 \right\}$ and $\mathcal{D}=\emptyset$. Therefore, $$\mathbb{Q}\text{-}\mathcal{KS}(N)=\mathcal{B}_{p,q}\cup\mathcal{C}_{p,q}\cup\left\{q+p-1 \right\}.$$
  \item [$(2)$] Assume that $3p<q<4p$. Then by ~\cite[Theorem $14$]{Ghanmi} we consider the following two subcases.
  \begin{enumerate}
    \item [$(a)$] Suppose that $q=4p-3$. Then $\mathcal{A}=\left\{q-p+1,q+p-1 \right\}$.

   Since  $\mathcal{D}=\left\{\dfrac{pq}{2p-1} \right\}$  by Lemma ~\ref{encadr17}  it follows that  $$\mathbb{Q}\text{-}\mathcal{KS}(N)=\mathcal{B}_{p,q}\cup \mathcal{C}_{p,q}\cup\left\{q-p+1, p+q-1, \dfrac{pq}{2p-1}\right\}.$$
    \item [$(b)$] If $q\neq4p-3$, then $\mathcal{A}=\left\{q+p-1 \right\}$. Also, since $\mathcal{D}=\emptyset$ by Lemma ~\ref{encadr17} we deduce that $$\mathbb{Q}\text{-}\mathcal{KS}(N)=\mathcal{B}_{p,q}\cup \mathcal{C}_{p,q}\cup\left\{ p+q-1 \right\}.$$
  \end{enumerate}

  \item [$(3)$] Now, suppose that  $2p<q<3p$. Then by Lemma ~\ref{encadr16} several subcases are to be discussed.
  \begin{enumerate}
         \item[(i)] If $s+1$ divides $q-1$ (which is equivalent to $\frac{s+1}{2}$ divides $p-1$).
                \begin{enumerate}
         \item [(a)]If $s=\dfrac{p-5}{4}$ or $s=p-2$, then  $\mathcal{D}=\left\{\dfrac{2pq}{3p-1}, \dfrac{pq}{2p-1}, \dfrac{2pq}{q+1} \right\}$ and by ~\cite[Lemma $12$]{Ghanmi},

             $$\mathcal{A}=\left\{3q-5p+3,q-p+1,p+q-1,\dfrac{2p+q-1}{2}\right\}.$$   Thus, we conclude that
              $$\mathbb{Q}\text{-}\mathcal{KS}(N)=\mathcal{B}_{p,q}\cup \mathcal{C}_{p,q}\cup\left\{3q-5p+3,q-p+1,p+q-1,\dfrac{2p+q-1}{2},\dfrac{2pq}{3p-1}, \dfrac{pq}{2p-1}, \dfrac{2pq}{q+1}\right\}.$$
         \item[(b)]If   $s\neq\dfrac{p-5}{4}$ and $s\neq p-2$ , then $\mathcal{D}=\left\{\dfrac{pq}{2p-1},\dfrac{2pq}{q+1} \right\}.$
         Now, let $ \mathcal{S}=\left\{q-p+1,p+q-1,\dfrac{2p+q-1}{2}\right\}$. We will prove that $\mathcal{A}=\mathcal{S}$.

         Since  $s+1\mid q-1$, it follows  by definition that  $\mathcal{S}\subseteq \mathcal{A}$.
           Also, as $s\neq\dfrac{p-5}{4}$, it follows according to  ~\cite[Lemma $12$]{Ghanmi}, that
           $\mathcal{A}\subseteq\mathcal{S}$ and so $\mathcal{A}=\mathcal{S}.$
         Thus, $$\mathbb{Q}\text{-}\mathcal{KS}(N)=\mathcal{B}_{p,q}\cup \mathcal{C}_{p,q}\cup \left\{q-p+1,p+q-1,\dfrac{2p+q-1}{2},\dfrac{pq}{2p-1},\dfrac{2pq}{q+1}\right\}.$$
         \end{enumerate}
         \item [(ii)]Now, suppose that $s+1$ doesn't divide $q-1$. Then $\mathcal{D}=\emptyset$ and   by  ~\cite[Lemma $12$]{Ghanmi} we have $\mathcal{A}= \left\{ p+q-1\right\}$. So, $$\mathbb{Q}\text{-}\mathcal{KS}(N)=\mathcal{B}_{p,q}\cup \mathcal{C}_{p,q}\cup \left\{p+q-1\right\}.$$
       \end{enumerate}
  \item [$(4)$] Assume that $q<2p$. Then by Lemmas~\ref{encadr6}, ~\ref{encadr7} and ~\ref{encadr8} we have $\mathcal{A}=\mathcal{A}_{p,q}$.
  Further, $\mathcal{C}=\mathcal{C}_{p,q}$ by  Lemma ~\ref{encadr14} and by   lemmas ~\ref{encadr19} and ~\ref{encadr30} we have  $\mathcal{D}=\mathcal{D}_{p,q}$.

    Now, let us determine  the set $\mathcal{B}$. First,  we remark by Lemma ~\ref{encadr11} and Proposition~\ref{encadr10} that \begin{equation}\label{eq95}\mathcal{B}_{p,q}\subseteq\mathcal{B}\subseteq\mathcal{B}_{p,q}\cup\left\{\dfrac{(s-1)q}{q-2},\dfrac{q}{2},\dfrac{q}{3}\right\}.
  \end{equation}

  Also, we need to state the next four assertions which can be verified easily, in order to  determine completely $\mathcal{B}$.

     \begin{itemize}
       \item If $\dfrac{q}{3}\in\mathcal{B}$ then $\dfrac{(s-1)q}{q-2}\in\mathcal{B}$.
       \item $ \left\{\dfrac{(s-1)q}{q-2},\dfrac{q}{2}\right\}\subseteq \mathcal{B}$ if and only if $s=p-1$.
       \item $ \left\{\dfrac{q}{2},\dfrac{q}{3}\right\}\subseteq \mathcal{B}$ if and only if $q=5$.
       \item If $q=5$ then $\dfrac{(s-1)q}{q-2}=\dfrac{q}{3}$.
     \end{itemize}
     This leads us to consider the following cases.
      \begin{enumerate}
      \item [(i)] If $q=5$, then $\mathcal{B}=\mathcal{B}_{p,q}\cup\left\{\dfrac{q}{2},\dfrac{q}{3}\right\}.$
      Thus, $$\mathbb{Q}\text{-}\mathcal{KS}(N)= \mathcal{A}_{p,q}\cup \mathcal{B}_{p,q}\cup \mathcal{C}_{p,q}\cup \mathcal{D}_{p,q}\cup \left\{\dfrac{q}{2},\dfrac{q}{3}\right\}.$$
      \item [(ii)]Suppose that $q\neq5$. Then the following subcases hold.
      \begin{enumerate}
      \item [(a)]If $p>2$ and $s=\dfrac{p+1}{2}$, then $\dfrac{q}{3}\in\mathcal{B}$ and so $\dfrac{(s-1)q}{q-2}\in\mathcal{B}$.

       Since in addition $q\neq5$, we have $\dfrac{q}{2}\notin\mathcal{B}$. Therefore $\mathcal{B}=\mathcal{B}_{p,q}\cup\left\{\dfrac{q}{3}\right\}.$  This implies that

       $\mathbb{Q}\text{-}\mathcal{KS}(N)= \mathcal{A}_{p,q}\cup \mathcal{B}_{p,q}\cup \mathcal{C}_{p,q}\cup \mathcal{D}_{p,q}\cup \left\{\dfrac{q}{3}\right\}$.
        \item [(b)]If $s=p-1$, then   $\left\{\dfrac{(s-1)q}{q-2},\dfrac{q}{2}\right\}\subseteq \mathcal{B}.$

        Since in addition $q\neq5$, we have $\dfrac{q}{3}\notin\mathcal{B}$. Therefore $\mathcal{B}=\mathcal{B}_{p,q}\cup\left\{\dfrac{q}{2},\dfrac{(s-1)q}{q-2}\right\}.$ So, we deduce that

      $\mathbb{Q}\text{-}\mathcal{KS}(N)= \mathcal{A}_{p,q}\cup \mathcal{B}_{p,q}\cup \mathcal{C}_{p,q}\cup \mathcal{D}_{p,q}\cup \left\{\dfrac{(s-1)q}{q-2},\dfrac{q}{2}\right\}.$

      \item[(c)] Assume that $(p-s)\mid(2s-1)$ and $s\neq p-1$. Then by Proposition ~\ref{encadr10} we have $\dfrac{q}{2}\in \mathcal{B}$.
       Also, as $p-s$  not dividing $s-1$ (else $s=p-1$), it follows that $\dfrac{(s-1)q}{q-2}\notin \mathcal{B}$ and as $q\neq5$, we have $\dfrac{q}{3}\notin \mathcal{B}$.   Consequently, we get by  $\eqref{eq95}$, $\mathcal{B}=\mathcal{B}_{p,q}\cup\left\{\dfrac{q}{2}\right\}$. This implies that

     $\mathbb{Q}\text{-}\mathcal{KS}(N)= \mathcal{A}_{p,q}\cup \mathcal{B}_{p,q}\cup \mathcal{C}_{p,q}\cup \mathcal{D}_{p,q}\cup \left\{\dfrac{q}{2}\right\}.$
      \item[(d)] If $(p-s)\mid(s-1)$, $s\neq\dfrac{p+1}{2}$ and $s\neq p-1$, then we get respectively $\dfrac{(s-1)q}{q-2}\in \mathcal{B}$, $\dfrac{q}{3}\notin \mathcal{B}$ and $\dfrac{q}{2}\notin \mathcal{B}$. Hence, $\mathcal{B}=\mathcal{B}_{p,q}\cup\left\{\dfrac{(s-1)q}{q-2}\right\}$, therefore

       $\mathbb{Q}\text{-}\mathcal{KS}(N)= \mathcal{A}_{p,q}\cup \mathcal{B}_{p,q}\cup \mathcal{C}_{p,q}\cup \mathcal{D}_{p,q}\cup \left\{\dfrac{(s-1)q}{q-2}\right\}.$
      \item[(e)] Now, suppose that $p-s$  divides neither $s-1$ nor $2s-1$. Then by Lemma~\ref{encadr11} and Proposition~\ref{encadr10} each of the rational numbers $\dfrac{(s-1)q}{q-2},\dfrac{q}{2}$ and $\dfrac{q}{3}$ is  not in  $\mathcal{B}$. Hence,  $\mathcal{B}=\mathcal{B}_{p,q}$. So, we conclude that

           $\mathbb{Q}\text{-}\mathcal{KS}(N)= \mathcal{A}_{p,q}\cup \mathcal{B}_{p,q}\cup \mathcal{C}_{p,q}\cup \mathcal{D}_{p,q}.$
       \end{enumerate}
    \end{enumerate}
\end{enumerate}
\end{proof}

\noindent \textbf{Acknowledgement.} I thank the referee for his/her
report improving both the presentation and the mathematical content
of the paper.

\end{document}